\documentclass[11pt]{article}
\usepackage[T1]{fontenc}
\usepackage{amsmath,amssymb,amsthm}
\usepackage[margin=1in]{geometry}
\usepackage[hidelinks]{hyperref}
\usepackage[expansion=false]{microtype}

\newtheorem{theorem}{Theorem}[section]
\newtheorem{proposition}[theorem]{Proposition}
\newtheorem{lemma}[theorem]{Lemma}
\newtheorem{corollary}[theorem]{Corollary}
\theoremstyle{remark}
\newtheorem{remark}[theorem]{Remark}
\newtheorem*{inputtheorem}{Smooth-domain input}
\newcommand{\D}{\mathbb D}
\newcommand{\R}{\mathbb R}
\newcommand{\Om}{\Omega}
\newcommand{\dist}{\operatorname{dist}}

\newcommand{\loc}{\mathrm{loc}}

\title{Removing Boundary Smoothness from the Hot-Spots Conclusion\\
in the Simple-Eigenvalue Case}
\author{Krzysztof Burdzy}
\date{October 7, 2026}

\begin{document}
\maketitle

\begin{abstract}
We give a reduction from a smooth-domain theorem excluding interior critical
points of first positive Neumann eigenfunctions to a non-strict hot-spots
conclusion on arbitrary bounded simply connected planar domains with compact
Neumann resolvent and simple first positive eigenvalue. The main analytic
step is spectral convergence under exhaustion by conformal images of
concentric disks. A reflection extension, conformal invariance of Dirichlet
energy, and compactness of the Sobolev embedding give a direct proof.
A truncation argument justifies square integrability of the transferred
extensions without any boundary regularity. The smooth-domain theorem is
an external input; the approximation argument is proved here in full.
The paper was generated by ChatGPT with prompts from the listed author.
\end{abstract}

\section{Human commentary}

(i) This paper was generated by ChatGPT in response to multiple prompts
from the human author. KB has not checked every detail, because the
theorem is, to a human, an obvious corollary of the results
in \cite{OpenAI} and can certainly be proved easily in several ways.

(ii) Independently and at the same time, Rodrigo Ba\~nuelos used ChatGPT
to generate an almost identical theorem and proof. Neither KB nor Rodrigo Ba\~nuelos intend
to submit their papers to a journal. If the results or arguments
are cited, equal credit should be given to KB, Rodrigo Ba\~nuelos and to
any other mathematicians who obtained the result at the same time.

(iii) The OpenAI model that generated \cite{OpenAI} missed the theorem
presented in this note. Several mathematicians studying the hot spots
conjecture for simply connected domains have considered nonsmooth
domains, including domains with internal cusps.

(iv) It is highly surprising that \cite{OpenAI} did not cite the
spectacular paper of de Dios Pont \cite{deDiosPont}, which shows that
the hot spots conjecture fails for convex domains in high dimensions.
Even if there were a reason for the omission, such as a gap in the
proof, the paper should nevertheless have been cited, with an
appropriate caveat.

\section{Statement and smooth-domain input}

Throughout, a domain is a nonempty connected open set. Let
$\Om\subset\R^2$ be bounded and simply connected. The Neumann Laplacian
$A_{\Om}$ is the nonnegative self-adjoint operator associated with the form
\[
 \mathcal E_{\Om}(f,g)=\int_{\Om}\nabla f\cdot\nabla g\,dx,
 \qquad \mathcal D(\mathcal E_{\Om})=H^1(\Om).
\]
No pointwise interpretation of a normal derivative on $\partial\Om$ is
required. We assume that
\begin{equation}\label{eq:compact}
 H^1(\Om)\hookrightarrow L^2(\Om)\quad\text{is compact}.
\end{equation}
Equivalently, $A_{\Om}$ has compact resolvent. Its eigenvalues, counted with
multiplicity, can then be written as
\[
 0=\lambda_0(\Om)<\lambda_1(\Om)\leq\lambda_2(\Om)\leq\cdots,
 \qquad \lambda_j(\Om)\longrightarrow\infty.
\]
Connectedness makes the zero eigenvalue simple. In particular,
\begin{equation}\label{eq:rayleigh}
 \lambda_1(\Om)=
 \min\left\{\int_{\Om}|\nabla f|^2\,dx:
 f\in H^1(\Om),\ \int_{\Om}f\,dx=0,\ \int_{\Om}f^2\,dx=1\right\}.
\end{equation}
All eigenfunctions below are real-valued and represented by their smooth
interior representatives.

The following is the external result used in the reduction. It is the
smooth-domain assertion stated in the preprint listed in the OpenAI
mathematical collection \cite{OpenAI}.

\begin{inputtheorem}\label{input:smooth}
If $G\subset\R^2$ is a bounded simply connected domain with smooth boundary,
then every nonzero eigenfunction $v$ for $\lambda_1(G)$ satisfies
\[
 \nabla v(x)\ne0\qquad (x\in G).
\]
\end{inputtheorem}

We take this assertion as an input and do not assess or reproduce its proof.
Only its consequence that $v$ has no interior local maximum or minimum is
needed. Thus every hot-spots conclusion below is an implication from this
explicit smooth-domain input. The spectral-convergence results do not
depend on it.

\begin{theorem}\label{thm:main}
Assume the smooth-domain input. Let $\Om\subset\R^2$ be bounded and simply
connected, suppose \eqref{eq:compact} holds, and suppose $\lambda_1(\Om)$
is simple. For every nonzero eigenfunction $u$ for $\lambda_1(\Om)$,
\begin{align}
 \sup_{x\in\Om}u(x)
 &=\lim_{\varepsilon\downarrow0}
   \sup_{\substack{x\in\Om\\\dist(x,\partial\Om)<\varepsilon}}u(x),
   \label{eq:hot}\\
 \inf_{x\in\Om}u(x)
 &=\lim_{\varepsilon\downarrow0}
   \inf_{\substack{x\in\Om\\\dist(x,\partial\Om)<\varepsilon}}u(x).
   \label{eq:cold}
\end{align}
The equalities are understood in the extended real numbers if $u$ is
unbounded. No regularity assumption on $\partial\Om$ is imposed.
\end{theorem}

\begin{corollary}\label{cor:continuous}
Under the assumptions of Theorem~\ref{thm:main}, if $u$ extends continuously
to $\overline\Om$, then
\[
 \max_{\overline\Om}u=\max_{\partial\Om}u,
 \qquad
 \min_{\overline\Om}u=\min_{\partial\Om}u.
\]
\end{corollary}

The theorem concerns extremal values approached near the boundary. It does
not assert that an interior point cannot share an extremal value with the
boundary, or that $\nabla u$ is nonzero throughout $\Om$.

\section{Conformal exhaustion and uniform smallness of exterior mass}

Fix a conformal bijection $\phi:\D\to\Om$, where
$\D=\{z\in\mathbb C:|z|<1\}$, and set
\[
 \D_r=\{z:|z|<r\},\qquad \Om_r=\phi(\D_r),\qquad 0<r<1.
\]
The map $\phi$ is holomorphic and injective on a neighborhood of
$\overline{\D_r}$, with nonvanishing derivative there. Consequently
$\Om_r$ has a real-analytic Jordan boundary. Furthermore,
\begin{equation}\label{eq:exhaust}
 \Om_r\Subset\Om,\qquad \Om_r\uparrow\Om\quad(r\uparrow1),
\end{equation}
and every compact subset of $\Om$ lies in $\Om_r$ for all sufficiently
large $r$. No extension of $\phi$ to $\partial\D$ is used.

\begin{lemma}\label{lem:tightness}
Under \eqref{eq:compact}, define
\[
 \kappa(r)=\sup_{0\ne f\in H^1(\Om)}
 \frac{\displaystyle\int_{\Om\setminus\Om_r}|f|^2\,dx}
 {\displaystyle\int_{\Om}(|f|^2+|\nabla f|^2)\,dx}.
\]
Then $\kappa(r)\to0$ as $r\uparrow1$.
\end{lemma}

\begin{proof}
Otherwise there are $r_n\uparrow1$, $c>0$, and $f_n\in H^1(\Om)$ such that
\[
 \|f_n\|_{H^1(\Om)}=1,
 \qquad \int_{\Om\setminus\Om_{r_n}}|f_n|^2\,dx\geq c.
\]
By compactness, after passing to a subsequence, $f_n\to f$ strongly in
$L^2(\Om)$. Hence
\[
 \int_{\Om\setminus\Om_{r_n}}|f_n|^2\,dx
 \leq 2\|f_n-f\|_{L^2(\Om)}^2
      +2\int_{\Om\setminus\Om_{r_n}}|f|^2\,dx
 \longrightarrow0.
\]
The last integral tends to zero by dominated convergence and
\eqref{eq:exhaust}, giving a contradiction.
\end{proof}

\section{An extension lemma, including square integrability}

\begin{lemma}\label{lem:extension}
For every $r\geq3/4$ sufficiently close to $1$ and every $f\in H^1(\Om_r)$,
there is $F\in H^1(\Om)$ equal to $f$ on $\Om_r$ such that
\begin{align}
 \int_{\Om}|\nabla F|^2\,dx
 &\leq3\int_{\Om_r}|\nabla f|^2\,dx,\label{eq:energy-extension}\\
 \int_{\Om\setminus\Om_r}|F|^2\,dx
 &\leq\frac{\kappa(r)}{1-\kappa(r)}
 \left(\int_{\Om_r}|f|^2\,dx
       +3\int_{\Om_r}|\nabla f|^2\,dx\right).
 \label{eq:mass-extension}
\end{align}
In particular, extensions of a family bounded in $H^1(\Om_r)$ are bounded
in $H^1(\Om)$ and have exterior $L^2$ mass tending to zero as $r\uparrow1$.
\end{lemma}

\begin{proof}
Put $q=f\circ\phi$ on $\D_r$. For each fixed $r<1$, the map $\phi$ and its
inverse have bounded derivatives on the relevant compact sets, so
$q\in H^1(\D_r)$. Extend $q$ to $\D$ by radial reflection:
\[
 Q(\rho e^{i\theta})=
 \begin{cases}
 q(\rho e^{i\theta}),&0\leq\rho<r,\\
 q((2r-\rho)e^{i\theta}),&r<\rho<1.
 \end{cases}
\]
The traces match at $\rho=r$, so $Q\in H^1(\D)$. This construction and
the calculation below can first be performed for smooth $q$ and then
extended by density.

In the reflected annulus put $s=2r-\rho$. Since $r\geq3/4$,
\[
 2r-1<s<r,\qquad \frac{2r-s}{s}\leq2,
 \qquad \frac{s}{2r-s}\leq1.
\]
The polar-coordinate energy formula therefore gives
\begin{align*}
 \int_{\D\setminus\D_r}|\nabla Q|^2
 &=\int_0^{2\pi}\int_{2r-1}^{r}
 \left((2r-s)|q_s|^2+\frac{|q_\theta|^2}{2r-s}\right)\,ds\,d\theta\\
 &\leq2\int_0^{2\pi}\int_{2r-1}^{r}
 \left(s|q_s|^2+\frac{|q_\theta|^2}{s}\right)\,ds\,d\theta
 \leq2\int_{\D_r}|\nabla q|^2.
\end{align*}
Thus
\begin{equation}\label{eq:disk-extension}
 \int_{\D}|\nabla Q|^2\leq3\int_{\D_r}|\nabla q|^2.
\end{equation}

Define $F=Q\circ\phi^{-1}$ on $\Om$. At this stage we know only that
$F\in H^1_{\loc}(\Om)$, that $F=f$ on $\Om_r$, and, by conformal
invariance of the two-dimensional Dirichlet integral, that
\[
 \int_{\Om}|\nabla F|^2
 =\int_{\D}|\nabla Q|^2
 \leq3\int_{\Om_r}|\nabla f|^2.
\]
Square integrability on $\Om$ does not follow merely from
$Q\in H^1(\D)$, because the change of variables introduces the possibly
unbounded weight $|\phi'|^2$. We establish it next.

For $M>0$ let $T_M(t)=\max(-M,\min(t,M))$ and set $F_M=T_M(F)$.
The function $F_M$ is bounded on the finite-measure set $\Om$, belongs to
$H^1_{\loc}(\Om)$, and has square-integrable weak gradient. Hence
$F_M\in H^1(\Om)$. Write
\[
 A=\int_{\Om_r}|f|^2,\qquad
 B=\int_{\Om_r}|\nabla f|^2,\qquad
 a_M=\int_{\Om\setminus\Om_r}|F_M|^2.
\]
Using Lemma~\ref{lem:tightness} and the Sobolev chain rule,
\[
 a_M\leq\kappa(r)\int_{\Om}(|F_M|^2+|\nabla F_M|^2)
 \leq\kappa(r)(A+a_M+3B).
\]
For $r$ sufficiently close to $1$, $\kappa(r)<1$. Rearranging yields
\[
 a_M\leq\frac{\kappa(r)}{1-\kappa(r)}(A+3B).
\]
Letting $M\to\infty$ and using monotone convergence proves
\eqref{eq:mass-extension}. In particular, $F\in L^2(\Om)$ and hence
$F\in H^1(\Om)$. The energy estimate was already established.
\end{proof}

\section{Convergence of the first positive eigenvalue and eigenfunctions}

\begin{proposition}\label{prop:spectral}
Let $\Om$ be bounded and simply connected and satisfy \eqref{eq:compact}.
Then
\begin{equation}\label{eq:eigenvalue-convergence}
 \lambda_1(\Om_r)\longrightarrow\lambda_1(\Om)
 \qquad(r\uparrow1).
\end{equation}
For any $r_n\uparrow1$ and any first eigenfunctions $u_n$ on $\Om_{r_n}$
normalized by
\begin{equation}\label{eq:normalization}
 \int_{\Om_{r_n}}u_n\,dx=0,
 \qquad \int_{\Om_{r_n}}u_n^2\,dx=1,
\end{equation}
there are a subsequence and a normalized first eigenfunction $v$ on $\Om$
such that
\begin{equation}\label{eq:local-convergence}
 u_n\longrightarrow v\quad\text{in }C^\infty_{\loc}(\Om).
\end{equation}
If $\lambda_1(\Om)$ is simple and $u$ is a prescribed normalized first
eigenfunction on $\Om$, the subsequence can be signed so that $v=u$.
\end{proposition}

\begin{proof}
Choose a normalized first eigenfunction $u$ on $\Om$. Define
\[
 c_r=\frac1{|\Om_r|}\int_{\Om_r}u\,dx.
\]
Then $c_r\to0$ and
\[
 \int_{\Om_r}|u-c_r|^2\,dx\longrightarrow1,
 \qquad
 \int_{\Om_r}|\nabla u|^2\,dx\longrightarrow\lambda_1(\Om).
\]
Testing the Rayleigh quotient with $u-c_r$ gives
\begin{equation}\label{eq:limsup}
 \limsup_{r\uparrow1}\lambda_1(\Om_r)\leq\lambda_1(\Om).
\end{equation}

For the opposite inequality, choose a sequence $r_n\uparrow1$ along which
$\lambda_1(\Om_{r_n})$ tends to its lower limit, and choose $u_n$ as in
\eqref{eq:normalization}. The eigenvalue equation gives
\[
 \int_{\Om_{r_n}}|\nabla u_n|^2\,dx=\lambda_1(\Om_{r_n}),
\]
which is bounded by \eqref{eq:limsup}. Lemma~\ref{lem:extension} provides
extensions $U_n\in H^1(\Om)$ for which
\begin{equation}\label{eq:extension-bounds}
 \sup_n\|U_n\|_{H^1(\Om)}<\infty,
 \qquad \int_{\Om\setminus\Om_{r_n}}|U_n|^2\,dx\longrightarrow0.
\end{equation}
After passage to a subsequence, compactness yields
\[
 U_n\rightharpoonup v\text{ in }H^1(\Om),
 \qquad U_n\longrightarrow v\text{ in }L^2(\Om).
\]
Since $U_n=u_n$ on $\Om_{r_n}$, \eqref{eq:normalization} and
\eqref{eq:extension-bounds} imply
\begin{equation}\label{eq:limit-normalization}
 \int_{\Om}v\,dx=0,\qquad\int_{\Om}v^2\,dx=1.
\end{equation}
For the mean, for example, Cauchy--Schwarz gives
\[
 \left|\int_{\Om}U_n\,dx\right|
 =\left|\int_{\Om\setminus\Om_{r_n}}U_n\,dx\right|
 \leq|\Om\setminus\Om_{r_n}|^{1/2}
      \|U_n\|_{L^2(\Om\setminus\Om_{r_n})}\longrightarrow0.
\]

Fix $s<1$. For all sufficiently large $n$, $\Om_s\subset\Om_{r_n}$.
Weak lower semicontinuity gives
\[
 \int_{\Om_s}|\nabla v|^2\,dx
 \leq\liminf_n\int_{\Om_s}|\nabla u_n|^2\,dx
 \leq\liminf_n\lambda_1(\Om_{r_n}).
\]
Letting $s\uparrow1$ and using \eqref{eq:rayleigh} and
\eqref{eq:limit-normalization}, we obtain
\[
 \lambda_1(\Om)\leq\int_{\Om}|\nabla v|^2\,dx
 \leq\liminf_n\lambda_1(\Om_{r_n}).
\]
Together with \eqref{eq:limsup}, this proves
\eqref{eq:eigenvalue-convergence}. Equality also shows that $v$ is a
minimizer in \eqref{eq:rayleigh}, hence a first Neumann eigenfunction.
In particular it satisfies the full weak Neumann equation
\[
 \int_{\Om}\nabla v\cdot\nabla h\,dx
 =\lambda_1(\Om)\int_{\Om}vh\,dx
 \qquad(h\in H^1(\Om)).
\]

Once eigenvalue convergence is known, the same compactness argument
applies to an arbitrary sequence $r_n\uparrow1$. On each relatively compact
subdomain of $\Om$, the functions $u_n$ eventually satisfy
\[
 -\Delta u_n=\lambda_1(\Om_{r_n})u_n.
\]
Their local $L^2$ norms and eigenvalues are uniformly bounded. Interior
elliptic estimates of every order, Sobolev embedding, and a diagonal
subsequence therefore give \eqref{eq:local-convergence}; the limit agrees
with the strong $L^2$ limit. Finally, if $\lambda_1(\Om)$ is simple,
normalization forces $v=u$ or $v=-u$. Changing signs gives $v=u$.
\end{proof}

\section{Passage of the hot-spots conclusion to the limit}

\begin{proof}[Proof of Theorem~\ref{thm:main}]
Normalize $u$ in $L^2(\Om)$. By Proposition~\ref{prop:spectral} and
simplicity, there are $r_n\uparrow1$ and first eigenfunctions $u_n$ on
$\Om_{r_n}$ such that
\[
 u_n\longrightarrow u\quad\text{locally uniformly in }\Om.
\]

First consider a possible unbounded supremum. For every $\delta>0$, the set
\[
 K_\delta=\{x\in\Om:\dist(x,\partial\Om)\geq\delta\}
\]
is a compact subset of $\Om$, so $u$ is bounded on $K_\delta$. If
$\sup_{\Om}u=+\infty$, its supremum in every boundary layer must therefore
also be $+\infty$, and \eqref{eq:hot} follows immediately.

We may consequently assume $M=\sup_{\Om}u<\infty$. Suppose
\eqref{eq:hot} fails. The boundary-layer suprema have a limit by
monotonicity. There are then $\delta>0$ and $\eta>0$ such that
\begin{equation}\label{eq:gap}
 u(x)\leq M-3\eta
 \quad\text{whenever }x\in\Om\text{ and }
 \dist(x,\partial\Om)<\delta.
\end{equation}
It follows that $M$ is attained on $K_\delta$, say at $x_0$.

Set $K=K_{\delta/2}$. Then $K\Subset\Om$, $x_0\in\operatorname{int}K$,
and every $x\in\partial K$ has $\dist(x,\partial\Om)=\delta/2$.
Thus \eqref{eq:gap} implies
\[
 u(x_0)=M,\qquad \max_{\partial K}u\leq M-3\eta.
\]
For sufficiently large $n$, $K\subset\Om_{r_n}$ and uniform convergence
on $K$ gives
\[
 u_n(x_0)>M-\eta,\qquad
 \max_{\partial K}u_n<M-2\eta.
\]
The continuous function $u_n$ attains its maximum over $K$ at a point
$y_n\in\operatorname{int}K$. This is an interior local maximum in
$\Om_{r_n}$, so $\nabla u_n(y_n)=0$. Since $\Om_{r_n}$ is smooth and
simply connected, this contradicts the smooth-domain input. Hence
\eqref{eq:hot} holds. Applying the same argument to $-u$ proves
\eqref{eq:cold}.
\end{proof}

\begin{proof}[Proof of Corollary~\ref{cor:continuous}]
Continuity on the compact set $\overline\Om$ implies uniform continuity.
Consequently the limiting supremum in shrinking interior boundary layers
equals $\max_{\partial\Om}u$: one inequality follows by taking subsequential
limits of points in the layers, and the other by approaching a maximizing
boundary point from within $\Om$. The analogous statement holds for the
minimum. Apply Theorem~\ref{thm:main}.
\end{proof}

\section{Scope of the argument}

\begin{remark}[Where simplicity is used]
Neither Lemma~\ref{lem:extension} nor Proposition~\ref{prop:spectral}
requires simplicity. Simplicity is used only to ensure that a prescribed
first eigenfunction on $\Om$ is, up to sign, the limit of first
eigenfunctions on the approximating domains. Without simplicity, the same
argument proves \eqref{eq:hot}--\eqref{eq:cold} for at least one nonzero
element of the first eigenspace. It does not establish them for every
element of that eigenspace.
\end{remark}

\begin{remark}[Which smooth-domain assertion is needed]
The maximum of $u_n$ over $K$ in the proof need not be its global maximum
over $\Om_{r_n}$. Thus the proof uses the absence of interior local
extrema supplied by the stronger smooth-domain input. The mere assertion
that global extremal values occur on the boundary of each smooth
approximating domain would not, by itself, justify this step under only
local uniform convergence.
\end{remark}

\begin{remark}[Relation to reflected Brownian motion]
Burdzy and Chen \cite{BC} establish weak convergence of reflecting Brownian
motions under increasing-domain approximation, subject to the hypotheses
and interpretations specified there. That result motivates the exhaustion
approach. The present proof instead works directly with Dirichlet energy
and compactness and establishes the needed spectral convergence without
invoking a stochastic convergence theorem or inferring spectral convergence
from strong resolvent convergence alone.
\end{remark}

\end{document}